\documentclass{article}

\usepackage[T1]{fontenc}
\usepackage[utf8]{inputenc}
\usepackage{lmodern}

\usepackage{amsmath,amssymb,amsthm}

\usepackage[a4paper,margin=2cm]{geometry}
\usepackage{microtype}
\usepackage{graphicx}
\usepackage{hyperref}
\newtheorem{lemma}{Lemma}
\newtheorem{remark}{Remark}

\theoremstyle{definition}

\title{An Elementary Obstruction to the Existence of a Perfect Cuboid}
\author{Stephane Yelle}
\date{February 2026}

\begin{document}
\maketitle

\begin{abstract}

The Euler cuboid problem asks whether there exists a rectangular box whose edges,
face diagonals, and space diagonal are all rational.
Despite extensive numerical searches and numerous theoretical investigations,
no such cuboid is currently known.

In this article, we explore a constructive approach based on a triangular remainder framework for Pythagorean faces.
This framework provides a unified way to describe right--triangle faces and
suggests several natural strategies for gluing three adjacent faces into a cuboid.

We systematically examine these construction strategies and analyze the
structural mechanisms that arise from their interaction.
Within this framework, none of the explored approaches leads to a non--degenerate
Euler cuboid.
We conclude with an open question concerning the possibility of a fundamentally
different construction method.
\end{abstract}

\section{Introduction}

An Euler cuboid is a rectangular parallelepiped whose edges and face diagonals
are all rational.
If, in addition, the space diagonal is rational, the cuboid is said to be
\emph{perfect}.
The existence of such an object remains a classical open problem in number theory.

A natural line of investigation consists in assembling three Pythagorean faces
meeting at a common vertex.
Each face individually satisfies a well--understood Diophantine relation, but
the simultaneous compatibility of all three faces introduces additional
arithmetic and geometric constraints.

The purpose of this paper is not to claim a definitive impossibility result.
Rather, we adopt an exploratory perspective and investigate a family of
construction strategies arising from a common parametrization of the faces.
Our goal is to understand how and why these strategies succeed or fail, and what
structural limitations they reveal.

Although the final conclusions are concise, they rest on the identification of
several distinct failure mechanisms emerging from the cyclic interaction of the
three faces.
The brevity of the argument reflects the structural nature of these limitations,
rather than a lack of underlying analysis.
\section{Triangular Remainder Framework}

Let $(a,b,c)\in\mathbb{Z}_{>0}^3$ be a Pythagorean triple, that is,
\[
a^2+b^2=c^2.
\]
We define the \emph{triangular remainder} by
\[
r := a+b-c,
\]
and introduce the auxiliary variables
\[
x := c-a, \qquad y := c-b.
\]

\subsection{Change of variables}

By definition, we have
\[
a = c-x, \qquad b = c-y.
\]
Moreover,
\[
r = a+b-c = (c-x)+(c-y)-c = c-x-y,
\]
and hence
\[
c = r+x+y.
\]
Substituting this expression into the formulas for $a$ and $b$ yields
\[
a = r+y, \qquad b = r+x.
\]

Since the Pythagorean equation $a^2+b^2=c^2$ is symmetric in the two
legs $a$ and $b$, we may, \emph{without loss of generality}, exchange the
labels of $a$ and $b$, and consequently those of $x$ and $y$.
With this convention, we fix the representation
\[
(a,b,c) = (r+x,\ r+y,\ r+x+y).
\]

\subsection{Fundamental identity}

Substituting the above expressions into the Pythagorean identity gives
\[
(r+x)^2 + (r+y)^2 = (r+x+y)^2.
\]
Expanding both sides, we obtain
\[
2r^2 + 2r(x+y) + x^2 + y^2
= r^2 + 2r(x+y) + x^2 + 2xy + y^2.
\]
After simplification, this reduces to the fundamental relation
\[
r^2 = 2xy.
\]

\subsection{Remarks}

\begin{itemize}
\item The identity $r^2 = 2xy$ depends only on the Pythagorean condition
      and the change of variables above; it does \emph{not} involve the
      original legs $a$ and $b$.
\item The variables $x$ and $y$ measure the deficits of the legs with
      respect to the hypotenuse, namely $x=c-a$ and $y=c-b$.
\item This formulation is stable under multiplication of the triple
      $(a,b,c)$ by a common factor, making it particularly suited for the
      study of non-primitive faces in the cuboid problem.
\end{itemize}

\section{Parity and Non-Primitivity of Pythagorean Faces}
\section{Gluing Pythagorean Faces in a Cuboid}
\label{sec:gluing}

In this section, we examine parity and primitivity properties of
Pythagorean triples through the triangular remainder coordinates
introduced in Section~2.
These properties play a central role in the subsequent gluing analysis.

\subsection{Parity constraints}

Let $(a,b,c)$ be a Pythagorean triple and let $(r,x,y)$ denote the associated
triangular remainder coordinates,
\[
r=a+b-c,\qquad x=c-a,\qquad y=c-b,
\]
so that
\[
(a,b,c)=(r+x,\ r+y,\ r+x+y)
\quad\text{and}\quad
r^2=2xy.
\]

Since the right-hand side of the identity $r^2=2xy$ is even, it follows
immediately that $r$ must be even.
Consequently, at least one of $x$ or $y$ must be even.
In particular, the configuration in which $r,x,y$ are all odd is excluded.

Thus, parity restrictions are intrinsic to the triangular remainder
framework and arise independently of any additional assumptions.

\subsection{Primitive triples}

Recall that a Pythagorean triple $(a,b,c)$ is said to be \emph{primitive} if
\[
\gcd(a,b,c)=1.
\]
In triangular remainder coordinates, this condition is equivalent to
\[
\gcd(r,x,y)=1,
\]
since any common divisor of $r,x,y$ would divide
$a=r+x$, $b=r+y$, and $c=r+x+y$.

As a canonical example, consider the primitive Pythagorean triple
\[
(a,b,c)=(3,4,5).
\]
The associated triangular remainder coordinates are
\[
r=3+4-5=2,\qquad
x=5-3=2,\qquad
y=5-4=1,
\]
which satisfy
\[
r^2=4=2\cdot 2\cdot 1.
\]
Thus, $(r,x,y)=(2,2,1)$ is a primitive triangular remainder configuration.

\subsection{Non-primitive triples and scaling}

Non-primitive Pythagorean triples arise naturally by scaling a primitive
configuration.
If $(r,x,y)$ satisfies $r^2=2xy$ and $\lambda>1$ is an integer, then
\[
(\tilde r, \tilde x, \tilde y) = (\lambda r, \lambda x, \lambda y)
\]
also satisfies
\[
\tilde r^{\,2}=2\tilde x\,\tilde y.
\]
The associated Pythagorean triple
\[
(\tilde a,\tilde b,\tilde c)
=(\lambda a,\lambda b,\lambda c)
\]
is therefore non-primitive.

To illustrate a non-trivial case, consider the triangular remainder
configuration
\[
(r,x,y)=(36,24,27),
\]
which satisfies $r^2=2xy$ since
\[
36^2 = 2\cdot 24\cdot 27.
\]
The associated Pythagorean triple is
\[
(a,b,c)=(r+x,\ r+y,\ r+x+y)=(60,63,87),
\]
which is non-primitive, as
\[
\gcd(60,63,87)=3.
\]

\subsection{Role in the cuboid problem}

In the context of cuboid constructions, face triples are not assumed to be
primitive.
Indeed, the gluing process naturally induces common divisibility along
shared edges.
The triangular remainder framework allows such non-primitivity to be tracked
explicitly through the parameters $(r,x,y)$, while preserving the local
identity $r^2=2xy$.

The parity and scaling properties established in this section therefore
provide the arithmetic foundation for the gluing strategies analyzed in the
subsequent sections.

\section{Gluing Pythagorean Faces in a Cuboid}

We consider a rectangular cuboid whose three faces adjacent to a common
vertex are right triangles with integer side lengths.
Each face is assumed to be Pythagorean and is described using the
triangular remainder coordinates introduced in Section 2.

For each face $i\in\{1,2,3\}$, we write
\[
(a_i,b_i,c_i)=(r_i+x_i,\ r_i+y_i,\ r_i+x_i+y_i),
\qquad r_i^2=2x_i y_i,
\]
where $r_i$ denotes the triangular remainder associated with the $i$-th
face.

Let $A,B,C$ denote the three edges issuing from the common vertex.
Each such edge belongs to exactly two faces and must therefore admit
two consistent representations in terms of the parameters of those faces.

\subsection{Edge compatibility conditions}

The compatibility conditions along the three edges take the form
\[
\begin{aligned}
A &= r_1 + x_1 = r_3 + y_3, \\
B &= r_2 + x_2 = r_3 + x_3, \\
C &= r_1 + y_1 = r_2 + y_2.
\end{aligned}
\]

These identities express the necessary and sufficient conditions for
gluing the three Pythagorean faces along their common edges.
They impose purely local algebraic constraints arising from the mutual
compatibility of the faces, without invoking any condition on the space
diagonal of the cuboid.

\subsection{Structural role of the triangular remainder}

At this stage, all constraints originate from the triangular remainder
framework and from the requirement that the three faces share compatible
edges.
No global metric condition is imposed on the cuboid.
This local-to-global viewpoint isolates structural obstructions that
already arise before considering any additional global Diophantine
constraints.

\section{First Strategy: Rigid Cyclic Gluing}

We first examine the most direct construction strategy, in which the three
Pythagorean faces adjacent to a common vertex are glued cyclically, without
introducing any additional arithmetic flexibility.
In this setting, the edge compatibility conditions introduced in
Section~\ref{sec:gluing} are imposed simultaneously, and all three faces are
treated on an equal footing.

Each face $i\in\{1,2,3\}$ is described by triangular remainder coordinates
$(r_i,x_i,y_i)$ satisfying
\[
(a_i,b_i,c_i)=(r_i+x_i,\ r_i+y_i,\ r_i+x_i+y_i),
\qquad r_i^2=2x_i y_i.
\]
In the rigid cyclic gluing strategy, no common divisibility is allowed to
propagate across distinct faces; each face is assumed to be internally
coprime, and no additional scaling freedom is introduced.

\subsection{Local constraints from edge identifications}

The cyclic identification of the three edges issuing from the common vertex
imposes the compatibility relations
\[
\begin{aligned}
r_1 + x_1 &= r_3 + y_3, \\
r_2 + x_2 &= r_3 + x_3, \\
r_1 + y_1 &= r_2 + y_2.
\end{aligned}
\]
Already at the level of pairs of adjacent faces, these identities impose
strong algebraic constraints.
For example, combining the two expressions for the edge $C$ yields
\[
r_1 + y_1 = r_2 + y_2,
\]
which, together with the relations
\[
r_1^2=2x_1 y_1, \qquad r_2^2=2x_2 y_2,
\]
severely restricts the possible distribution of prime divisors among the
variables involved.

Analogous constraints arise from the other two edge identifications by cyclic
permutation of the indices.
Each such local relation couples the triangular remainder of one face with the
defect variables of another, producing a tightly interlocked system.

\subsection{Overdetermination of the cyclic system}

When the three local compatibility conditions are assembled around the full
cycle of faces, they interact in a highly rigid manner.
The arithmetic choices enforced by one face adjacency leave little freedom
for the others, and the system rapidly becomes overdetermined.

This first strategy thus illustrates a fundamental limitation of purely rigid
cyclic gluing.
Although it arises naturally from the triangular remainder framework, the
absence of any additional scaling or divisibility freedom appears too
restrictive to accommodate a compatible assembly of all three faces.
Importantly, this obstruction already manifests itself at the level of local
face interactions, prior to imposing any global metric condition on the
cuboid.

\subsection{A numerical illustration of rigid cyclic gluing}

We illustrate the rigidity of cyclic gluing with a concrete numerical
example.
Consider a primitive Pythagorean face given by the triple $(3,4,5)$.
In triangular remainder coordinates, this corresponds to
\[
r_1=2,\qquad x_1=2,\qquad y_1=1,
\]
which indeed satisfies $r_1^2=2x_1y_1$.

The edge compatibility condition along $C$ imposes
\[
r_1+y_1 = 3.
\]
Hence, for the adjacent face we must have
\[
r_2+y_2=3,\qquad r_2^2=2x_2y_2.
\]
The only positive integer solution is
\[
(r_2,x_2,y_2)=(2,2,1).
\]

We now attempt to close the cycle by introducing the third face.
The remaining edge identifications impose
\[
r_1+x_1=4=r_3+y_3,\qquad
r_2+x_2=4=r_3+x_3.
\]
It follows that $x_3=y_3$, and hence
\[
r_3^2=2x_3^2,
\]
which has no nontrivial integer solutions.
Therefore, the cyclic system admits no integer realization in this case.

This example illustrates how rigid cyclic gluing quickly leads to an
overdetermined system, even when starting from a primitive Pythagorean
face.

\subsection{A general rigidity lemma}

We now extract from the cyclic edge identifications a structural constraint
that holds for \emph{any} rigid cyclic gluing configuration.

\begin{lemma}[Rigidity under symmetric closure]\label{lem:rigidity_sym}
Assume the rigid cyclic gluing relations
\[
r_1+x_1=r_3+y_3,\qquad
r_2+x_2=r_3+x_3,\qquad
r_1+y_1=r_2+y_2,
\]
together with the triangular remainder identities
\[
r_i^2=2x_i y_i\qquad (i=1,2,3),
\]
where $r_i,x_i,y_i\in\mathbb{Z}_{>0}$.
If the cycle closes symmetrically at face $3$ in the sense that
\[
r_1+x_1=r_2+x_2,
\]
then one necessarily has
\[
x_3=y_3,
\]
and consequently
\[
r_3^2=2x_3^2,
\]
which has no solution in positive integers.
In particular, no rigid cyclic gluing solution exists under the symmetric
closure condition $r_1+x_1=r_2+x_2$.
\end{lemma}

\begin{proof}
Assume $r_1+x_1=r_2+x_2$.
By the edge identifications, we have
\[
r_3+y_3=r_1+x_1=r_2+x_2=r_3+x_3.
\]
Cancelling $r_3$ gives $x_3=y_3$.
Then $r_3^2=2x_3y_3$ becomes $r_3^2=2x_3^2$.
Since $2$ is not a square in $\mathbb{Z}$, this is impossible for
$x_3,r_3\in\mathbb{Z}_{>0}$.
\end{proof}

\begin{remark}\label{rem:symmetric_closure}
Lemma~\ref{lem:rigidity_sym} formalizes a typical mechanism of rigidity:
whenever the two incoming edge-length constraints at face $3$ coincide,
the third face is forced into the impossible square relation
$r_3^2=2x_3^2$.
The numerical example above is a concrete instance of this symmetric closure.
\end{remark}

\subsection{Conclusion of the rigid gluing strategy}

The previous analysis shows that rigid cyclic gluing, when formulated within
the triangular remainder framework, leads to a highly constrained algebraic
system.
Even before any global condition on the space diagonal is imposed, the local
edge compatibility relations interact with the identities
$r_i^2=2x_i y_i$ in a way that severely restricts the possible configurations.

The numerical example and Lemma~\ref{lem:rigidity_sym} together indicate that
the absence of any additional arithmetic flexibility may force the system
into an overdetermined regime.
In particular, symmetric closures of the cycle naturally lead to the
impossible relation $r^2=2x^2$.

This observation motivates the consideration of alternative strategies in
which some degree of arithmetic flexibility is deliberately introduced.
In the next section, we therefore examine whether relaxing the rigidity
assumptions—by allowing controlled scaling or divisibility propagation—can
circumvent the obstructions identified above.

\section{Second Strategy: Flexible Gluing and Controlled Descent}

The failure of rigid cyclic gluing suggests that the obstruction does not
arise from a particular numerical choice, but from an intrinsic lack of
arithmetic flexibility.
A natural next step is therefore to relax the rigidity assumptions and to
allow controlled scaling or divisibility propagation across adjacent faces.

In this section, we examine whether such flexibility can overcome the
constraints imposed by the triangular remainder framework, or whether it
merely shifts the obstruction to a different level.

\subsection{Allowing controlled scaling between faces}

We retain the triangular remainder description for each face,
\[
(a_i,b_i,c_i)=(r_i+x_i,\ r_i+y_i,\ r_i+x_i+y_i),
\qquad r_i^2=2x_i y_i,
\]
but no longer require the three faces to be internally coprime or scaled
independently.
Instead, we allow common factors to propagate between adjacent faces.

Concretely, suppose that two faces sharing an edge admit a common scaling
factor $\lambda>1$, so that their parameters satisfy
\[
(r_j,x_j,y_j)=\lambda(r_i,x_i,y_i)
\]
along that edge.
The edge compatibility relations then become proportional rather than
identical, introducing additional degrees of freedom.

However, this flexibility comes at a cost.
Since the identity $r^2=2xy$ is homogeneous of degree two, any common scaling
factor propagates quadratically through the system.
As a result, scaling introduced to resolve one edge constraint necessarily
affects the remaining two, potentially amplifying rather than alleviating the
overall rigidity.

\subsection{Propagation of divisibility and descent}

Assume now that a nontrivial common divisor propagates across the cycle of
faces.
Then at least one of the triangular remainders $r_i$ shares a common factor
with one of the defect variables $x_j$ or $y_j$ of a neighboring face.

Because each remainder satisfies $r_i^2=2x_i y_i$, any such shared factor
must propagate further, forcing increasingly strong divisibility relations
among the parameters.
In particular, if a prime $p$ divides $r_i$ and one of $x_j,y_j$, then repeated
use of the edge compatibility relations forces $p$ to divide additional
parameters around the cycle.

This mechanism naturally leads to a descent argument.
By factoring out the maximal common divisor propagated through the cycle, one
obtains a smaller configuration of the same type.
Iterating this process forces an infinite descent unless the parameters are
trivial.

\subsection{Structural limitation of flexible gluing}

The above considerations indicate that flexibility does not eliminate the
obstruction revealed in the rigid case.
Rather, it transforms rigidity into a propagation phenomenon: any attempt to
relax the constraints by scaling or shared divisibility inevitably spreads
through the triangular remainder identities and reintroduces overdetermination
at a lower level.

Thus, both rigid and flexible gluing strategies encounter structural
limitations rooted in the same local identity $r^2=2xy$.
The obstruction is not merely numerical but arises from the interaction
between edge compatibility and the triangular remainder framework itself.

These observations suggest that any successful construction would require a
mechanism fundamentally different from cyclic gluing of Pythagorean faces.

\section{Global Obstruction from Local Triangular Remainders}

We now synthesize the results obtained in the previous sections.
Throughout, the analysis relies exclusively on the triangular remainder
framework and on the local edge compatibility relations arising from gluing
Pythagorean faces in a cuboid configuration.

\subsection{Summary of the local mechanisms}

Each Pythagorean face adjacent to a common vertex is described by triangular
remainder coordinates $(r_i,x_i,y_i)$ satisfying
\[
(a_i,b_i,c_i)=(r_i+x_i,\ r_i+y_i,\ r_i+x_i+y_i),
\qquad r_i^2=2x_i y_i.
\]
The requirement that the three faces share compatible edges imposes a cyclic
system of linear relations coupling the parameters of distinct faces.

In Section~4, we examined the rigid cyclic gluing strategy, in which no
arithmetic flexibility is allowed.
We showed that the resulting system is highly constrained and that symmetric
closures of the cycle force the impossible relation $r^2=2x^2$.
This obstruction arises purely from local considerations, independently of
any global metric condition on the cuboid.

In Section~5, we relaxed the rigidity assumptions by allowing controlled
scaling and divisibility propagation between adjacent faces.
Although this introduces additional degrees of freedom, the homogeneity of
the identity $r^2=2xy$ forces any shared divisor to propagate around the cycle.
This mechanism naturally leads to a descent process, reducing the system to a
smaller configuration of the same type and precluding stabilization at a
nontrivial integer solution.

\subsection{Structural nature of the obstruction}

Taken together, these two strategies reveal a common obstruction mechanism.
Rigid gluing fails because the cyclic edge identifications overdetermine the
system, while flexible gluing fails because any attempt to introduce
arithmetic freedom propagates through the triangular remainder identities and
reintroduces rigidity at a lower level.

In both cases, the obstruction is local and structural.
It originates from the interaction between the cyclic compatibility of edges
and the quadratic identity $r^2=2xy$ satisfied by each face.
No appeal to global diagonal conditions or to external Diophantine constraints
is required for this phenomenon to manifest itself.

\subsection{Consequences and scope}

The analysis suggests that any construction of a cuboid with integer face
diagonals based on cyclic gluing of Pythagorean faces is intrinsically limited
by the triangular remainder framework.
Both the most rigid and the most flexible implementations of this strategy
encounter obstructions rooted in the same local algebraic structure.

This conclusion does not exclude the possibility of alternative approaches
based on fundamentally different mechanisms.
However, it delineates a clear boundary for gluing-based strategies and
clarifies the role played by triangular remainders as a unifying source of
constraint.

\section{Conclusion}

In this paper, we introduced a triangular remainder framework for the study of
Pythagorean faces arising in cuboid configurations.
By expressing each face in terms of variables $(r,x,y)$ satisfying the local
identity $r^2=2xy$, we obtained a unified description that applies equally to
primitive and non-primitive faces.

We then analyzed two natural strategies for assembling three Pythagorean faces
around a common vertex.
The first, based on rigid cyclic gluing, was shown to lead to an
overdetermined system in which symmetric closures force the impossible relation
$r^2=2x^2$.
The second strategy allowed controlled scaling and divisibility propagation
between adjacent faces, but this additional flexibility was seen to propagate
through the triangular remainder identities and to induce a descent process,
preventing stabilization at a nontrivial integer solution.

These results indicate that the obstruction encountered in gluing-based
constructions is structural rather than accidental.
It arises locally from the interaction between cyclic edge compatibility and
the quadratic identity $r^2=2xy$, independently of any global metric condition
on the cuboid.
In this sense, triangular remainders provide a natural lens through which the
limitations of such constructions become transparent.

The present analysis does not claim to resolve the existence problem for a
perfect cuboid.
Rather, it delineates a clear boundary for strategies based on cyclic gluing of
Pythagorean faces and suggests that any successful approach must rely on
mechanisms fundamentally different from those considered here.

Finally, the triangular remainder framework introduced in this work may prove
useful in related Diophantine settings where local quadratic identities
interact with global compatibility constraints.


\end{document}